\documentclass[12pt,reqno]{amsart}
\usepackage{amsmath,amssymb,amsfonts,amsthm}
\usepackage[left=3cm, right=3cm, top=2.8cm, bottom=2.8cm]{geometry}
\usepackage[utf8]{inputenc}
\usepackage[T1]{fontenc}
\usepackage{xcolor}
\usepackage{multirow}
\usepackage{graphicx}
\usepackage{hyperref}
\usepackage{graphicx}
\usepackage{calc}%
\usepackage{hyperref}
\hypersetup{
    colorlinks=true,
    linkcolor=blue,
    filecolor=magenta,
    urlcolor=black,
}

\usepackage{stackrel}

\DeclareMathOperator*{\esssup}{ess\,sup}
\newtheorem{theorem}{Theorem}
\newtheorem{lemma}[theorem]{Lemma}

\newtheorem{corollary}[theorem]{Corollary}
\newtheorem{definition}[theorem]{Definition}
\newtheorem{remark}[theorem]{Remark}

\theoremstyle{plain}

\begin{document}

\title[The Riemann-Liouville fractional integral]{The Riemann-Liouville fractional integral in Bochner-Lebesgue spaces IV}


\author[P. M. Carvalho-Neto]{Paulo M. de Carvalho-Neto}
\address[Paulo M. de Carvalho Neto]{Departamento de Matem\'atica, Centro de Ciências Físicas e Matemáticas, Universidade Federal de Santa Catarina, Florian\'{o}polis - SC, Brazil}
\email[]{paulo.carvalho@ufsc.br}
\author[R. Fehlberg J\'{u}nior]{Renato Fehlberg J\'{u}nior}
\address[Renato Fehlberg J\'unior]{Departamento de Matem\'atica, Universidade Federal do Esp\'{i}rito Santo, Vit\'{o}ria - ES, Brazil}
\email[]{renato.fehlberg@ufes.br}






\begin{abstract}
In this manuscript, we extend our previous work on the Riemann-Liouville fractional integral of order $\alpha > 0$ in Bochner-Lebesgue spaces. We specifically address the remaining cases concerning its boundedness when $\alpha > 1/p$. Furthermore, we extend some of our previous results by investigating some non-standard function spaces. Finally, we provide a comprehensive summary of the obtained results.
\end{abstract}

\maketitle

\section{An Overview of the Riemann-Liouville Fractional Integral}

Throughout this manuscript, we assume that $t_0 < t_1$ are fixed real numbers, $\mathbb{N}^\ast:=\{1,2,\ldots\}$, $p^\prime=p/(p-1)$ is the H\"older conjugate of $p$, and $X$ is a Banach space. Let us begin by revisiting the definitions and notations of specific classical vector spaces and operators. At first we recall the classical Bochner-Lebesgue spaces (for details, see \cite{ArBaHiNe1}).

\begin{definition} If $1\leq p\leq\infty$, we represent the set of all Bochner measurable functions $f:[t_0,t_1]\rightarrow X$, for which $\|f\|_X\in L^{p}(t_0,t_1)$, by the symbol $L^{p}(t_0,t_1;{X})$. Moreover, $L^{p}(t_0,t_1;{X})$ is a Banach space when considered with the norm
$$\|f\|_{L^p(t_0,t_1;X)}:=\left\{\begin{array}{ll}\bigg[\displaystyle\int_{t_0}^{t_1}{\|f(s)\|^p_X}\,ds\bigg]^{1/p},&\textrm{ if }p\in[1,\infty),\vspace*{0.3cm}\\
\esssup_{s\in [t_0,t_1]}\|f(s)\|_X,&\textrm{ if }p=\infty.\end{array}\right.$$
\end{definition}

Now let us recall the Bochner-Sobolev spaces (for more details see \cite{CaHa1}).

\begin{definition} Let $n\in\mathbb{N}^\ast$ and $1\leq p\leq\infty$. The set $W^{n,p}(t_0,t_1;X)$ denotes the subspace of $L^{p}(t_0,t_1;X)$ of every function $f:[t_0,t_1]\rightarrow X$ that has $n-$weak derivatives in $L^{p}(t_0,t_1;X)$. By considering the norm
$$\|f\|_{W^{n,p}(t_0,t_1;X)}:=\sum_{j=0}^n\big\|f^{(j)}\big\|_{L^{p}(t_0,t_1;X)},$$
the set $W^{n,p}(t_0,t_1;X)$ becomes a Banach space. For the completeness of the definition we assume that $W^{0,p}(t_0,t_1;X)=L^p(t_0,t_1;X)$.
\end{definition}

Now we present the notions of Riemann-Liouville fractional integral and derivative.

\begin{definition} Let $\alpha\in(0,\infty)$ and $f:[t_0,t_1]\rightarrow{X}$. The Riemann-Liouville (RL for short) fractional integral of order $\alpha$ at $t_0$ of function $f$ is defined by
\begin{equation}\label{fracinitt}J_{t_0,t}^\alpha f(t):=\dfrac{1}{\Gamma(\alpha)}\displaystyle\int_{t_0}^{t}{(t-s)^{\alpha-1}f(s)}\,ds,\end{equation}
for every $t\in [t_0,t_1]$ such that the integral at the right-hand side of \eqref{fracinitt} exists. Above $\Gamma$ denotes the classical Euler's gamma function.
\end{definition}

\begin{definition} Let $\alpha\in(0,\infty)$ and consider $f:[t_0,t_1]\rightarrow{X}$. The Riemann-Liouville (RL for short) fractional derivative of order $\alpha$ at $t_0$ of function $f$ is defined by
\begin{equation}\label{fracinit}D_{t_0,t}^\alpha f(t):=\dfrac{d^{[\alpha]}}{dt^{[\alpha]}}\left[J_{t_0,t}^{[\alpha]-\alpha} f(t)\right]=\dfrac{d^{[\alpha]}}{dt^{[\alpha]}}\left[\dfrac{1}{\Gamma([\alpha]-\alpha)}\displaystyle\int_{t_0}^{t}{(t-s)^{[\alpha]-\alpha-1}f(s)}\,ds\right],\end{equation}
for every $t\in [t_0,t_1]$ such that the right side of \eqref{fracinit} exists. The derivative above is considered in the weak sense, and $[\alpha]$ denotes the least integer greater than $\alpha$.
\end{definition}

\begin{remark} In this work, we frequently use two relationships between the RL fractional derivative and the corresponding fractional integral:
\begin{itemize}
\item[(i)] It is well known in the literature (cf. \cite[Proposition 2.11]{CarFe3}) that if $\alpha \in (0,1)$, $f \in L^1(t_0,t_1; X)$, and $J_{t_0,t}^{1-\alpha} f \in W^{1,1}(t_0,t_1; X)$, then
$$J_{t_0,t}^{\alpha}\big[D_{t_0,t}^\alpha f(f)\big]=f(t)-\dfrac{(t-t_0)^{\alpha-1}}{\Gamma(\alpha)}\Big[J_{t_0,s}^{1-\alpha}f(s)\big|_{s=t_0}\Big],$$
for almost every $t\in[t_0,t_1]$. Moreover, if $f\in C([t_0,t_1];X)$, the above identity simplifies to
\begin{equation}\label{auxlatul1}J_{t_0,t}^{\alpha}\big[D_{t_0,t}^\alpha f(t)\big]=f(t),\end{equation}
for almost every $t\in[t_0,t_1]$. Finally, if $f(t)$ and $D_{t_0,t}^{\alpha} f(t)$ belong to $C([t_0,t_1]; X)$, \cite[Proposition 4.6]{CarFe3} implies that $J_{t_0,t}^{\alpha} \big[D_{t_0,t}^{\alpha} f(t)\big]$ is continuous on $[t_0,t_1]$, \eqref{auxlatul1} holds for every $t \in [t_0,t_1]$, and
$$J_{t_0,t}^{\alpha}\big[D_{t_0,t}^\alpha f(t)\big]\big|_{t=t_0}=0.$$
Therefore, we conclude that $f(t_0)=0$.\vspace*{0.2cm}
\item[(ii)] If $f\in W^{1,1}(t_0,t_1;X)$, $f(t_0)=0$, and $\alpha\in(0,1)$, it holds that
$$\dfrac{d}{dt}\Big[J_{t_0,t}^{\alpha}f(t)\Big]=J_{t_0,t}^{\alpha}f^\prime(t),$$
for almost every $t\in[t_0,t_1]$; see \cite[Remark 2.10]{CarFe3} for details.
\end{itemize}
\end{remark}

With the above notions in mind, let us recall that in our three recent works \cite{CarFe0, CarFe1, CarFe2}, we have conducted an extensive investigation into the Riemann-Liouville fractional integral of order $\alpha > 0$ when it is considered as a linear operator from $L^p(t_0, t_1; X)$ into $L^q(t_0, t_1; X)$ (or another suitable function space), with $X$ being a Banach space. 

In this sequence of works, our focus was on establishing the boundedness and, when possible, compactness of this operator, along with some related properties. This study led to several noteworthy conclusions; however, it is important to emphasize that not all cases have been addressed so far. This is the main reason why, in this final work, we intend not only to address these remaining cases but also to provide an overview of the Riemann-Liouville fractional integral as a bounded operator. 

We begin by providing a comprehensive list of all our results on the boundedness of the Riemann-Liouville fractional integral, including their references and the specific cases addressed in this manuscript.

\begin{itemize}
\item[(i)] For $p=1$ we have:\vspace*{0.2cm}
\begin{itemize}
\item[(a$_1$)] if $\alpha\in(0,1)$, then $J_{t_0,t}^\alpha:L^1(t_0,t_1;X)\rightarrow L^q(t_0,t_1;X)$ is a bounded linear operator for any $q\in\big[1,1/(1-\alpha)\big)$. Moreover, for any $r\in [1/(1-\alpha),\infty]$ it holds that $J_{t_0,t}^\alpha\big(L^p(t_0,t_1;X)\big)\not\subset L^r(t_0,t_1;X)$; see \cite[{Theorem 7}]{CarFe1}.\vspace*{0.2cm}
\item[(a$_2$)] if $\alpha\in(0,1)$, then $J_{t_0,t}^\alpha:L^1(t_0,t_1;X)\rightarrow L_w^{1/(1-\alpha)}(t_0,t_1;X)$ is a ``bounded'' linear operator; see {\cite[Corollary 1]{CarFe1}}. For details on the space $L_w^{r}(t_0,t_1;X)$, see Definition \ref{weakdef} or {\cite[Remark 4]{CarFe1}}. Furthermore, for any $r\in \big(1/(1-\alpha),\infty\big]$ we have that $J_{t_0,t}^\alpha\big(L^p(t_0,t_1;X)\big)\not\subset L_w^r(t_0,t_1;X)$; see Theorem \ref{comp01}. \vspace*{0.2cm}
\item[(b)] if $\alpha \in [1,\infty)$ and $\gamma \in (0,\alpha]$, then $J_{t_0,t}^\alpha:L^1(t_0,t_1;X)\rightarrow W_{RL}^{\gamma,1}(t_0,t_1;X)$ is a bounded linear operator; see Theorem \ref{consobfrac} and {\color{black}\cite[Theorem 10]{CarFe2}}. Moreover, for any $r_1 \in [\alpha,\infty)$ and $r_2 \in [1,\infty)$, with $r_1 + r_2 > 1 + \alpha$, it holds that $J_{t_0,t}^\alpha\big(L^1(t_0,t_1;X)\big)\not\subset W_{RL}^{r_1,r_2}(t_0,t_1;X)$; see {\color{black}\cite[Theorem 11]{CarFe2}}. For details on $W_{RL}^{\alpha,1}(t_0,t_1;X)$, see Definition \ref{sobolevriemann11}. For the definition and properties of $W_{RL}^{r_1,r_2}(t_0,t_1;X)$, see {\color{black}\cite[Definition 7 and Proposition 8]{CarFe2}}. \vspace*{0.2cm}
\end{itemize}

\item[(ii)] For $p\in(1,\infty)$ we have:\vspace*{0.2cm}
\begin{itemize}
\item[(a)] if $\alpha\in(0,1/p)$, then $J_{t_0,t}^\alpha:L^p(t_0,t_1;X)\rightarrow L^q(t_0,t_1;X)$ is a bounded linear operator for any $q\in\big[1,p/(1-p\alpha)\big]$. Moreover, for any $r\in (p/(1-p\alpha),\infty]$ it holds that $J_{t_0,t}^\alpha\big(L^p(t_0,t_1;X)\big)\not\subset L^r(t_0,t_1;X)$; see \cite[{Theorem 6}]{CarFe1}.\vspace*{0.2cm}
\item[(b$_1$)] if $\alpha=1/p$ and $q\geq 1$, then $J_{t_0,t}^{1/p}:L^p(t_0,t_1;X)\rightarrow L^q(t_0,t_1;X)$ is a bounded linear operator. Moreover, it holds that $J_{t_0,t}^{1/p}\big(L^p(t_0,t_1;X)\big)\not\subset L^\infty(t_0,t_1;X)$; see {\color{black}\cite[Theorem 28]{CarFe2}}.\vspace*{0.2cm}
\item[(b$_2$)] if $\alpha=1/p$ and $\gamma\geq 1/p^\prime$, then $J_{t_0,t}^{1/p}:L^p(t_0,t_1;X)\rightarrow {BMO}(t_0,t_1;X)\cap K_{\gamma}(t_0,t_1;X)$ is a bounded linear operator; see {\color{black}\cite[Theorem 31]{CarFe2}}. For details on the spaces ${BMO}(t_0,t_1;X)$ and $K_{\gamma}(t_0,t_1;X)$, see Definitions \ref{bmospace} and \ref{karaspace} or {\color{black}\cite[Section 4: Some special spaces]{CarFe2}}.\vspace*{0.2cm}
\item[(c)] if $\alpha\in(1/p,\infty)$, then $J_{t_0,t}^\alpha:L^p(t_0,t_1;X)\rightarrow C\big([t_0,t_1];X\big)$ is a bounded linear operator; see Theorem \ref{continitial01}.
\item[(d)] if $n\in\mathbb{N}$, $\alpha\in\big(n+(1/p),n+1+(1/p)\big)$ and $q\in\big(0,\alpha-n-(1/p)\big]$, then $J_{t_0,t}^\alpha:L^p(t_0,t_1;X)\rightarrow H^{n,q}([t_0,t_1];X)$ is a bounded linear operator; see Definition \ref{holderfunction}, Theorem \ref{lebsguecontinuity2} and Corollary \ref{lebsguecontinuity3}. Moreover, if $r\in\big(\alpha-n-(1/p),1\big)$, then $J_{t_0,t}^\alpha(L^p(t_0,t_1;X))\not\subset H^{n,r}([t_0,t_1];X)$;  see Remark \ref{lebsguecontinuity4}. \vspace*{0.2cm}
\item[(e)] if $n\in\mathbb{N}^\ast$, $\alpha=n+(1/p)$, and $\gamma\geq 1/p^\prime$, then $J_{t_0,t}^\alpha:L^p(t_0,t_1;X)\rightarrow K^{n,p}_{\gamma}(t_0,t_1;X)$ is a bounded linear operator; see Definition \ref{lebsguecontinuity6} for details on the space $K^{n,p}_{\gamma}(t_0,t_1;X)$ and Theorem \ref{lebsguecontinuity5} for the proof of the result.\vspace*{0.3cm}

\end{itemize}
\item[(iii)] For $p=\infty$ we have:\vspace*{0.2cm}
\begin{itemize}
\item[(a)] if $\alpha\in(0,1)$, then $J_{t_0,t}^\alpha:L^\infty(t_0,t_1;X)\rightarrow H^{0,\alpha}([t_0,t_1];X)$ is a bounded linear operator. Moreover, for any $r\in (\alpha,1)$ it holds that $J_{t_0,t}^\alpha\big(L^\infty(t_0,t_1;X)\big)\not\subset H^{0,r}([t_0,t_1];X)$; see Theorem \ref{infcase1}.\vspace*{0.2cm}
\item[(b)] if $\alpha> 1$, $\alpha\not\in\mathbb{N}$ and $q\in\big(0,\alpha-[\alpha]+1\big]$, then $J_{t_0,t}^\alpha:L^\infty(t_0,t_1;X)\rightarrow H^{[\alpha]-1,q}([t_0,t_1];X)$ is a bounded linear operator. Moreover, for any $r\in \big(\alpha-[\alpha]+1,1\big)$ it holds that $J_{t_0,t}^\alpha\big(L^\infty(t_0,t_1;X)\big)\not\subset H^{[\alpha]-1,r}([t_0,t_1];X)$; see Theorem \ref{infcase2}.
\item[(c)] if $\alpha\in\mathbb{N}^\ast$, then $J_{t_0,t}^\alpha:L^\infty(t_0,t_1;X)\rightarrow W^{\alpha,\infty}(t_0,t_1;X)$ is a bounded linear operator; see Theorem \ref{infcase2}.
\end{itemize}
\end{itemize}

This manuscript is organized as follows. Section \ref{sec3} addresses the continuity of the Riemann-Liouville (RL) fractional integral for $\alpha>1/p$ (for any $p>1$); then this result is improved for $\alpha \in (n+(1/p), n +1 + (1/p))$ and $\alpha = n + 1/p$, for $n\in\mathbb{N}^\ast$. In Section \ref{prim}, we enhance previous results: Subsection \ref{sub01} reviews weak Bochner-Lebesgue spaces and improves \cite[Corollary 1]{CarFe1}; Subsection \ref{sub02} revisits the spaces $W_{RL}^{\gamma,1}(t_0,t_1; X)$ which was originally introduced by Carbotti and Comi in \cite{CaCo1} (see also \cite{BoId1,IdWa1}) and extended to vector-valued functions in \cite{CarFe2}, enhancing \cite[Theorem 10]{CarFe2}; finally, Subsection \ref{sub03} addresses the continuity of the RL fractional integral on $L^\infty(t_0,t_1;X)$.

\section{Boundeness of the RL Fractional Integral - The Supercritical Case}
\label{sec3}

In \cite{HaLi1}, the classical paper by Hardy-Littlewood, the results \cite[Theorem 12 and 14]{HaLi1} deal with the case when the order $\alpha$ of the RL fractional integral is greater than $1/p$ and $X=\mathbb{R}$. In a certain sense, the adaptation of this theorems to the Bochner-Lebesgue spaces might appear intuitive. Nevertheless, we have opted for a more compelling approach, one that involves the introduction of simpler proofs, supported by novel arguments.

In order to address our objective, we recall the famous dominated-convergence theorem-like, which was proved by Brezis-Lieb (see \cite[Theorem 1]{BrLi1}).

\begin{lemma}\label{theoBL} Let $p>1$, $f\in L^p(t_0,t_1)$ and $\{f_k\}_{k=1}^\infty\subset L^p(t_0,t_1)$. If $f_k\rightarrow f$ almost everywhere and there exists $M>0$ such that $\|f_k\|_{L^p(t_0,t_1)}\leq M$, for every $k\in\mathbb{N}^\ast$, then $f\in L^p(t_0,t_1)$ and
$$\lim_{k\rightarrow\infty}{\Big(\|f_k\|_{L^p(t_0,t_1)}^p-\|f-f_k\|_{L^p(t_0,t_1)}^p\Big)}=\|f\|_{L^p(t_0,t_1)}^p.$$
\end{lemma}

With Lemma \ref{theoBL} in mind, we now present two results that replace \cite[Theorems 12 and 14]{HaLi1} in the context of Bochner-Lebesgue spaces.

\begin{theorem}\label{continitial01} Consider $p\in(1,\infty)$, $\alpha\in(1/p,\infty)$ and assume that $f\in L^p(t_0,t_1;X)$. Then $J_{t_0,t}^\alpha f\in C\big([t_0,t_1];X\big)$ and there exists $K=K(\alpha,p,t_0,t_1)>0$ such that
\begin{equation*}\sup_{t\in[t_0,t_1]}\|J^\alpha_{t_0,s}f(s)\|_X\leq K\left[\int_{t_0}^{t_1}{\|f(s)\|_X^p}\,ds\right]^{1/p}.\end{equation*}
\end{theorem}
\begin{proof} At first, assume that $\alpha<1$. Let $t \in (t_0,t_1]$ and consider the function $\phi_t:[t_0,t_1]\rightarrow\mathbb{R}$ defined by
$$\phi_t(s) = \left[\frac{(t-s)^{\alpha-1}}{\Gamma(\alpha)}\right] \chi_{[t_0,t]}(s),$$
where
$$\chi_{[t_0,t]}(s)=\left\{\begin{array}{ll}1,&\textrm{ if }s\in[t_0,t),\\0,&\textrm{ if }s\in[t,t_1].\end{array}\right.$$

Since $\phi_t \in L^{p^\prime}(t_0,t_1)$ and $f \in L^{p}(t_0,t_1; X)$, H\"{o}lder's inequality ensures that $\phi_t f \in L^1(t_0,t_1; X)$ and
$$\int_{t_0}^{t_1} \phi_t(s) f(s) \, ds = \frac{1}{\Gamma(\alpha)} \int_{t_0}^t (t-s)^{\alpha-1} f(s) \, ds,$$
which implies that $J^\alpha_{t_0,t} f(t)$ exists for every $t \in (t_0,t_1]$. Additionally, H\"{o}lder's inequality provides
\begin{multline*}
\big\| J_{t_0,t}^\alpha f(t) \big\|_X = \left\| \frac{1}{\Gamma(\alpha)} \int_{t_0}^t (t-s)^{\alpha-1} f(s) \, ds \right\|_X \\
\leq \frac{1}{\Gamma(\alpha)} \left( \int_{t_0}^t (t-s)^{p^\prime(\alpha-1)} \, ds \right)^{1/p^\prime} \left( \int_{t_0}^t \|f(s)\|_X^p \, ds \right)^{1/p},
\end{multline*}
for every $t \in (t_0,t_1]$, implying that
\begin{equation}\label{new01}\big\|J_{t_0,t}^\alpha f(t)\big\|_X\leq\left(\dfrac{(t-t_0)^{\alpha-(1/p)}}{\big\{p^\prime[\alpha-(1/p)]\big\}^{1/p^\prime}\Gamma(\alpha)}\right)\|f\|_{L^p(t_0,t_1;X)},\end{equation}
for every $t \in (t_0,t_1]$.

Next, we prove the continuity of $J_{t_0,t}^\alpha f(t)$ on $[t_0,t_1]$. Define $J_{t_0,t}^\alpha f(t)|_{t=t_0} = 0$. From \eqref{new01}, it follows that $J_{t_0,t}^\alpha f(t)$ is continuous at $t = t_0$. Now, fix $w \in (t_0,t_1)$ and observe that for any $\tau \in (w,t_1]$, we have
\begin{multline}\label{cont00}
\big\| J_{t_0,\tau}^\alpha f(\tau) - J_{t_0,w}^\alpha f(w) \big\|_X \leq \frac{1}{\Gamma(\alpha)} \int_{t_0}^w \left| (w-s)^{\alpha-1} - (\tau-s)^{\alpha-1} \right| \| f(s) \|_X \, ds \\
+ \frac{1}{\Gamma(\alpha)} \int_{w}^\tau (\tau-s)^{\alpha-1} \| f(s) \|_X \, ds=:\mathcal{H}_w(\tau)+\mathcal{I}_w(\tau).
\end{multline}

For $\mathcal{H}_w(\tau)$, we have
$$\big|(w-s)^{\alpha-1}-(\tau-s)^{\alpha-1}\big|\,\big\|f(s)\big\|_X\leq (w-s)^{\alpha-1}\,\big\|f(s)\big\|_X,$$
for almost every $s\in [t_0,w]$, with the function in the right side of the inequality being integrable in $[t_0,w]$, and that
$$\lim_{\tau\rightarrow w^+}{\big[(w-s)^{\alpha-1}-(\tau-s)^{\alpha-1}\big]\,\big\|f(s)\big\|_X}=0,$$
for almost every $s\in [t_0,w]$. Then, dominated convergence theorem guarantees that
\begin{equation}\label{cont01}\lim_{\tau\rightarrow w^+}{\mathcal{H}_w(\tau)}=0.\end{equation}
On the other hand, like it was done to deduce inequality \eqref{new01}, we obtain
\begin{equation*}\mathcal{I}_w(\tau)\leq\left(\dfrac{(\tau-w)^{\alpha-(1/p)}}{\big\{p^\prime[\alpha-(1/p)]\big\}^{1/p^\prime}\Gamma(\alpha)}\right)\|f\|_{L^p(t_0,t_1;X)},\end{equation*}
for almost every $\tau\in (w,t_1]$. Therefore
\begin{equation}\label{cont02}\lim_{\tau\rightarrow w^+}{\mathcal{I}_w(\tau)}=0.\end{equation}

Hence, items \eqref{cont00}, \eqref{cont01} and \eqref{cont02} ensure that
\begin{equation}\label{cont03}\lim_{\tau\rightarrow w^+}{J_{t_0,\tau}^\alpha f(\tau)}=J_{t_0,w}^\alpha f(w),\end{equation}
in the topology of $X$.

Now fix $w\in(t_0,t_1]$ and consider $\tau\in(t_0,w)$. Then
\begin{multline}\label{cont05}\big\|J_{t_0,\tau}^\alpha f(\tau)-J_{t_0,w}^\alpha f(w)\big\|_X\\\leq\dfrac{1}{\Gamma(\alpha)}{\int_{t_0}^w{\big|(\tau-s)^{\alpha-1}\chi_{[t_0,\tau]}(s)-(w-s)^{\alpha-1}\big|\,\big\|f(s)\big\|_X}\,ds}.\end{multline}

H\"older's inequality ensures that
\begin{multline*}\int_{t_0}^w{\big|(\tau-s)^{\alpha-1}\chi_{[t_0,\tau]}(s)-(w-s)^{\alpha-1}\big|\,\big\|f(s)\big\|_X}\,ds
\\\leq\underbrace{\left(\int_{t_0}^w{\big|(\tau-s)^{\alpha-1}\chi_{[t_0,\tau]}(s)-(w-s)^{\alpha-1}\big|^{p^\prime}}\,ds\right)^{1/p^\prime}}_{=\mathcal{J}_w(\tau)}\left(\int_{t_0}^w{\|f(s)\|_X^p}\,ds\right)^{1/p}.\end{multline*}

Observe that we may rewrite $\big[\mathcal{J}_w(\tau)\big]^{p^\prime}$ as
\begin{multline}\label{auxBL}\left[\int_{t_0}^w{\big|(\tau-s)^{\alpha-1}\chi_{[t_0,\tau]}(s)-(w-s)^{\alpha-1}\big|^{p^\prime}}\,ds
-\int_{t_0}^w{\big|(\tau-s)^{\alpha-1}\chi_{[t_0,\tau]}(s)\big|^{p^\prime}}\,ds\right]\\+\int_{t_0}^w{\big|(\tau-s)^{\alpha-1}\chi_{[t_0,\tau]}(s)\big|^{p^\prime}}\,ds.\end{multline}

Since we already know that
$$\int_{t_0}^w{\big|(\tau-s)^{\alpha-1}\chi_{[t_0,\tau]}(s)\big|^{p^\prime}}\,ds=\dfrac{(\tau-t_0)^{(\alpha-1)p^\prime+1}}{(\alpha-1)p^\prime+1}\leq\dfrac{(w-t_0)^{(\alpha-1)p^\prime+1}}{(\alpha-1)p^\prime+1},$$
for every $\tau\in[t_0,w)$, and that $(\tau-s)^{\alpha-1}\chi_{[t_0,\tau]}(s)\rightarrow (w-s)^{\alpha-1}$, when $\tau\rightarrow w^{-}$ almost everywhere, by applying Lemma \ref{theoBL} to equation \eqref{auxBL}, when $\tau\rightarrow w^-$, we deduce that the term in brackets tends to
$$-\int_{t_0}^w{\big|(w-s)^{\alpha-1}\big|^{p^\prime}}\,ds,$$
while the other term tends, when $\tau\rightarrow w^-$, to
$$\int_{t_0}^w{\big|(w-s)^{\alpha-1}\big|^{p^\prime}}\,ds.$$

Thus we conclude that the right term in \eqref{cont05} tends to zero, when $\tau\rightarrow w^-$, what is equivalent to
\begin{equation}\label{cont04}\lim_{\tau\rightarrow w^-}{J_{t_0,\tau}^\alpha f(\tau)}=J_{t_0,w}^\alpha f(w),\end{equation}
in the topology of $X$. Finally, by items \eqref{cont03} and \eqref{cont04}, $J_{t_0,t}^\alpha f\in C\big([t_0,t_1];X\big)$.

Finally, for $\alpha\geq 1$ we observe that 
$$J_{t_0,t}^\alpha f(t)=J_{t_0,t}^{[(p+1)/(2p)]}\Big\{J_{t_0,t}^{\alpha-[(p+1)/(2p)]} f(t)\Big\},$$
for almost every $t\in[t_0,t_1]$. Since \cite[Theorem 3.1]{CarFe0} ensures that $J_{t_0,t}^{\alpha-[(p+1)/(2p)]} f\in L^p(t_0,t_1;X)$ and $(1/p)<(p+1)/(2p)<1$, the first part of the proof of this theorem ensures that $J_{t_0,t}^\alpha f\in C\big([t_0,t_1];X\big)$.
\end{proof}

To present our next result, we need to recall the definition of the H\"{o}lder continuous functions.

\begin{definition}\label{holderfunction} If $n\in\mathbb{N}$ and $\gamma\in(0,1)$, we define $H^{n,\gamma}(I;X)$ as the space of the $n$-times continuously differentiable functions $f:I\rightarrow X$ such that
$$\big\|f^{(n)}(t)-f^{(n)}(s)\big\|_X\leq M{|t-s|^\gamma},$$
for every $t,s\in[t_0,t_1]$ and some constant $M>0$. If $I$ is a compact set, by introducing the norm
$$\|f\|_{H^{n,\gamma}(I;X)}:={\|f\|_{C^{n}(I;X)}+[\,f^{(n)}\,]_{H^{0,\gamma}(I;X)}},$$
where
$$[\,f^{(n)}\,]_{H^{0,\gamma}(I;X)}:=\sup_{t,s\in I,\,\,t\not=s}{\left[\dfrac{\big\|f^{(n)}(t)-f^{(n)}(s)\big\|_X}{|t-s|^\gamma}\right]},$$
$H^{n,\gamma}(I; X)$ becomes a Banach space. Here, $f^{(j)}(t)$ denotes the classical $j$-times derivative of $f(t)$. When $\gamma \in (0,1)$, we call these functions Hölder continuously differentiable functions of order $n$ with exponent $\gamma$.

\end{definition}

This definition allows us to establish the following relationship between the spaces $L^p(t_0,t_1;X)$, $H^{0,q}\big([t_0,t_1];X\big)$, and the Riemann-Liouville fractional integral.

\begin{theorem}\label{lebsguecontinuity2} Assume that $p\in(1,\infty)$ and $\alpha\in(1/p,1+(1/p))$. If $f\in L^p(t_0,t_1;X)$ and we consider $q=\alpha-(1/p)$, then $J_{t_0,t}^\alpha f\in H^{0,q}\big([t_0,t_1];X\big)$ and there exists a constant $K=K(\alpha,p,t_0,t_1)>0$ such that
\begin{equation*}\|J^\alpha_{t_0,s}f\|_{H^{0,q}([t_0,t_1];X)}\leq K\left[\int_{t_0}^{t_1}{\|f(s)\|_X^p}\,ds\right]^{1/p}.\end{equation*}
\end{theorem}
\begin{proof}First, let $t \in (t_0, t_1]$ and note that by H\"{o}lder's inequality, we have (recall inequality \eqref{new01})
\begin{multline}\label{novaeq01211}\big\|J^\alpha_{t_0,t}f(t)-J^\alpha_{t_0,w}f(w)\big|_{w=t_0}\big\|_X\leq\dfrac{1}{\Gamma(\alpha)}\int_{t_0}^t{(t-s)^{\alpha-1}\|f(s)\|_X}\,ds\\
\leq \left(\dfrac{(t-t_0)^{\alpha-(1/p)}}{\big\{p^\prime[\alpha-(1/p)]\big\}^{1/p^\prime}\Gamma(\alpha)}\right)\|f\|_{L^p(t_0,t_1;X)}.\end{multline}

Now, let $t, w \in (t_0, t_1]$ and, without loss of generality, assume that $t > w$. Then
\begin{multline*}
\big\|J^\alpha_{t_0,t} f(t) - J^\alpha_{t_0,w} f(w)\big\|_X \leq \dfrac{1}{\Gamma(\alpha)} \int_w^t (t-s)^{\alpha-1} \|f(s)\|_X \, ds \\
+ \dfrac{1}{\Gamma(\alpha)} \int_{t_0}^w \big|(w-s)^{\alpha-1} - (t-s)^{\alpha-1}\big| \|f(s)\|_X \, ds =: \mathcal{H} + \mathcal{I}.
\end{multline*}

Following the ideas used to obtain \eqref{new01}, we deduce that
\begin{equation}\label{novaeq01}\mathcal{H}\leq \left(\dfrac{(t-w)^{\alpha-(1/p)}}{\big\{p^\prime[\alpha-(1/p)]\big\}^{1/p^\prime}\Gamma(\alpha)}\right)\|f\|_{L^p(t_0,t_1;X)}.\end{equation}

Now, if $\alpha<1$ (the case $\alpha\geq1$ is analogous), H\"older's inequality ensures that
\begin{multline*}\mathcal{I}\leq\dfrac{1}{\Gamma(\alpha)}\left\{\int_{t_0}^w{\big[(w-s)^{\alpha-1}-(t-s)^{\alpha-1}\big]^{p^\prime}}\,ds\right\}^{1/p^\prime}\|f\|_{L^p(t_0,t_1;X)}\\
=\dfrac{1}{\Gamma(\alpha)}\left\{\int_{t_0}^w{\left[-\int_w^t{\dfrac{d}{dr}(r-s)^{\alpha-1}}\,dr\right]^{p^\prime}}\,ds\right\}^{1/p^\prime}\|f\|_{L^p(t_0,t_1;X)}\\
=\dfrac{(1-\alpha)}{\Gamma(\alpha)}\left\{\int_{t_0}^w{\left[\int_w^t{(r-s)^{\alpha-2}}\,dr\right]^{p^\prime}}\,ds\right\}^{1/p^\prime}\|f\|_{L^p(t_0,t_1;X)}.\end{multline*}
By applying Minkowski’s Inequality to Integrals (see \cite[Theorem 5.2]{CarFe0}), we deduce
\begin{multline*}\mathcal{I}\leq\dfrac{(1-\alpha)}{\Gamma(\alpha)}\left\{\int_{w}^t{\left[\int_{t_0}^w{(r-s)^{p^\prime(\alpha-2)}}\,ds\right]^{1/p^\prime}}\,dr\right\}\|f\|_{L^p(t_0,t_1;X)}\\
=\dfrac{(1-\alpha)}{\Gamma(\alpha)}\left\{\int_{w}^t{\left[\dfrac{(r-t_0)^{p^\prime(\alpha-2)+1}-(r-w)^{p^\prime(\alpha-2)+1}}{p^\prime(\alpha-2)+1}\right]^{1/p^\prime}}\,dr\right\}\|f\|_{L^p(t_0,t_1;X)}\\
\leq\dfrac{(1-\alpha)}{\Gamma(\alpha)}\left\{\int_{w}^t\left[\dfrac{(r-t_0)^{\alpha-(1/p)-1}}{\big(p^\prime(\alpha-2)+1\big)^{1/p^\prime}}\,\right]dr\right\}\|f\|_{L^p(t_0,t_1;X)},\end{multline*}
and therefore
\begin{equation}\label{novaeq02}\mathcal{I}\leq (1-\alpha)\left\{\dfrac{(t-t_0)^{\alpha-(1/p)}-(w-t_0)^{\alpha-(1/p)}}{\Gamma(\alpha)\big[\alpha-(1/p)\big]\big[p^\prime(\alpha-2)+1\big]^{1/p^\prime}}\right\}\|f\|_{L^p(t_0,t_1;X)}.\end{equation}

Since $\alpha-(1/p)\in(0,1)$, there exists $\tilde{M}=\tilde{M}(\alpha,t_0,t_1)>0$ such that
\begin{equation*}(t-t_0)^{\alpha-(1/p)}-(w-t_0)^{\alpha-(1/p)}\leq \tilde{M}(t-w)^{\alpha-(1/p)},\end{equation*}
what together with \eqref{novaeq01} and \eqref{novaeq02}, allow us to deduce the existence of a constant ${M}={M}(\alpha,p,t_0,t_1)>0$ such that
\begin{equation}\label{novaeq03}
\big\|J^\alpha_{t_0,t} f(t) - J^\alpha_{t_0,w} f(w)\big\|_X \leq M(t-w)^{\alpha-(1/p)}\|f\|_{L^p(t_0,t_1;X)},
\end{equation}
for every $t, w \in (t_0, t_1]$.

Finally, Theorem \ref{continitial01} and inequalities \eqref{novaeq01211} and \eqref{novaeq03} are enough for us to complete the proof of this theorem.
\end{proof}

As a direct corollary of Theorem \ref{lebsguecontinuity2}, we have:

\begin{corollary}\label{lebsguecontinuity3} Assume that $p\in(1,\infty)$ and $\alpha\in(n+(1/p),n+1+(1/p))$, $n\in\mathbb{N}$. If $f\in L^p(t_0,t_1;X)$ and we consider $q=\alpha-(n+(1/p))$, then $J_{t_0,t}^\alpha f\in H^{n,q}\big([t_0,t_1];X\big)$ and there exists a constant $K=K(\alpha,p,t_0,t_1)>0$ such that
\begin{equation*}\|J^\alpha_{t_0,s}f\|_{H^{n,q}([t_0,t_1];X)}\leq K\left[\int_{t_0}^{t_1}{\|f(s)\|_X^p}\,ds\right]^{1/p}.\end{equation*}
\end{corollary}

\begin{remark}\label{lebsguecontinuity4}
  Note that in Theorem \ref{lebsguecontinuity2} and Corollary \ref{lebsguecontinuity3} we can not improve the regularity.  For this purpose, let $p>1$ and $\alpha\in (1/p,1+1/p)$. Consider $q\in (\alpha-1/p,1)$ and choose $\gamma\in(-1/p,q-\alpha)$. Then $f(t)=t^\gamma\in L^p(0,1)$ and $J^\alpha_{0,t}f(t)=ct^{\gamma+\alpha}$, for some constant $c$. Since $\alpha+\gamma<q$, $J^\alpha_{0,t}f(t)\notin H^{0,q}([0,1])$. We can modify this example to the general case of Corollary \ref{lebsguecontinuity3}.
\end{remark}

The final part of this section addresses a case not discussed above, namely when $p > 1$ and $\alpha = n + (1/p)$ for some $n \in \mathbb{N}$. To handle this case, we need to recall the definitions of two spaces (cf. \cite{CarFe2}): the Bounded Mean Oscillation space and the Karapetyants-Rubin space.

\begin{definition}\label{bmospace} If $f\in L^1_{loc}(t_0,t_1;X)$ and
$$[f]_{BMO(t_0,t_1;X)}:=\sup_{[a,b]\subset (t_0,t_1)}{\left[\dfrac{1}{(b-a)}\int_a^b{\|f(s)-\operatorname{avg}_{[a,b]}(f)\|_X}\,ds\right]}<\infty,$$
where
$$\operatorname{avg}_{[a,b]}(f):=\dfrac{1}{(b-a)}\int_a^b{f(s)}\,ds,$$
we say that $f$ is a vector-valued function whose mean oscillation is bounded (finite). We denote the vector space of all those functions by $BMO(t_0,t_1;X)$. Above we use the symbol $L_{loc}^1(t_0, t_1; X)$ to denote the set of functions that belong to $L^1(K; X)$ for any compact $K\subset (t_0, t_1)$.
\end{definition}

\begin{definition}\label{karaspace}  For $\gamma > 0$, the Karapetyants-Rubin space (KR-space, for short) is the collection of functions $f \in \bigcap_{r \geq 1} L^r(t_0, t_1; X)$ that satisfy
$$\sup\{r^{-\gamma}\|f\|_{L^r(t_0,t_1;X)}:r\geq1\}<\infty.$$
We denote this space by $K_{\gamma}(t_0, t_1; X)$. If we consider
$$\|f\|_{K_{\gamma}(t_0, t_1; X)}:=\sup\{r^{-\gamma}\|f\|_{L^r(t_0,t_1;X)}:r\geq1\},$$
then it becomes a Banach space.
\end{definition}

\begin{remark} As discussed in {\color{black}\cite[Theorem 27]{CarFe2}}, for $0 < \gamma < 1$, the space $BMO(t_0, t_1; X) \cap K_\gamma(t_0, t_1; X)$ is a Banach space when equipped with the norm
$$\|\phi\|_\gamma := [\phi]_{BMO(t_0, t_1; X)} + \|\phi\|_{K_\gamma(t_0, t_1; X)}.$$
Moreover, {\color{black}\cite[Theorem 31]{CarFe2}} guarantees that for $p > 1$ and $\gamma \geq 1/p^\prime$, the operator $J^{1/p}_{t_0,t}: L^p(t_0, t_1; X) \rightarrow BMO(t_0,t_1;X) \cap K_\gamma(t_0, t_1; X)$ is bounded and linear, i.e., there exists a constant $K > 0$ such that
$$\|J_{t_0,t}^{1/p}f\|_\gamma \leq K \|f\|_{L^p(t_0, t_1; X)},$$
for every $f \in L^p(t_0, t_1; X).$
\end{remark}

With the aid of the above spaces and the theory established in \cite{CarFe2}, we introduce a new space that plays a crucial role in proving the continuity of the Riemann-Liouville fractional integral of order $n + (1/p)$ in the Bochner-Lebesgue space $L^p(t_0,t_1;X)$.

\begin{definition}\label{lebsguecontinuity6}
For $p> 1$, $\gamma>0$ and  $n\in\mathbb{N}^\ast$, we use the symbol $BK_\gamma^{n,p}(t_0,t_1; X)$ to denote the subspace of $W^{n,p}(t_0,t_1; X)$ consisting of all functions $f:[t_0,t_1]\rightarrow X$ such that $f^{(n)} \in {BMO}(t_0,t_1;X)\cap K_{\gamma}(t_0,t_1;X)$. By considering the norm
\begin{equation*}\|f\|_{BK_\gamma^{n,p}(t_0,t_1; X)}:=\|f\|_{W^{n,p}(t_0,t_1;X)}+\|f^{(n)}\|_{\gamma},\end{equation*}
it becomes a Banach space.
\end{definition}

\begin{remark} Given the theorems already established in \cite{CarFe2} and the completeness of the Bochner-Sobolev spaces, it is straightforward to show that $BK_\gamma^{n,p}(t_0,t_1; X)$ is also a Banach space with the norm defined above. For this reason, we have chosen to omit the proof.
\end{remark}

With the space $BK_\gamma^{n,p}(t_0,t_1; X)$ now introduced, we are ready to present the final result of this section, which completes our analysis of the critical case and is the main focus of this work.

\begin{theorem}\label{lebsguecontinuity5}
  Let $p\in(1,\infty)$ and assume that there exists $n\in\mathbb{N}^\ast$ such that $\alpha=n+1/p$. If $f\in L^p(t_0,t_1;X)$, then $J_{t_0,t}^\alpha f\in BK_\gamma^{n,p}(t_0,t_1; X)$, for $\gamma\geq 1/p^\prime$. Moreover, there exists $M>0$ such that
$$\|J_{t_0,t}^{1/p}f\|_{BK_\gamma^{n,p}(t_0,t_1; X)} \leq M \|f\|_{L^p(t_0, t_1; X)},$$
for every $f \in L^p(t_0, t_1; X).$
\end{theorem}

\begin{proof}
  Let  $f\in L^p(t_0,t_1;X)$. Observe that for $j\in\{0,1,\ldots,n\}$ we have that 
  $$\dfrac{d^j}{dt^j}\Big[J_{t_0,t}^\alpha f(t)\Big]=J_{t_0,t}^{n-j+(1/p)} f(t),$$
  for almost every $t\in[t_0,t_1]$. Now, \cite[Theorem 3.1]{CarFe0} and {\color{black}\cite[Theorem 31]{CarFe2}}) allow us to conclude that
  \begin{multline*}\|J_{t_0,t}^{1/p}f\|_{BK_\gamma^{n,p}(t_0,t_1; X)}\\
  \leq \left\{\sum_{j=0}^n\left[\dfrac{(t_1-t_0)^{n-j+(1/p)}}{\Gamma(n-j+(1/p)+1)}\right]+K\left[\dfrac{(t_1-t_0)^{(1/p)}}{\Gamma((1/p)+1)}\right]\right\}\left\|f\right\|_{L^{p}(t_0,t_1; X)},\end{multline*}
  for some constant $K>0$, what completes this proof.
  \end{proof}

\section{Complementary Results and Remarks}
\label{prim}

In this section, we present new results not covered in our previous works (see \cite{CarFe0, CarFe1, CarFe2}),
which are essential to the comprehensive framework introduced at the beginning of this manuscript.

\subsection{Improving \cite[Corollary 1]{CarFe1}}\label{sub01} 

Here we address the previously established relationship between the Riemann-Liouville fractional integral of order $\alpha$ and the weak $L^p$ spaces. 

\begin{definition}\label{weakdef} Consider $f:I\rightarrow X$ a Lebesgue measurable function. We define $\lambda_f:(0,\infty)\rightarrow[0,\infty]$, the distribution function of $f(t)$ by
$\lambda_f(r)=\mu\{t\in I:\| f(t)\|_X>r\}$,
where $\mu$ denotes the Lebesgue measure. Now, if $p\in[1,\infty)$ we define the weak $L^p(I;X)$, denoted by $L_w^p(I;X)$, to be the set of all Bochner measurable functions $f:I\rightarrow X$ such that
\begin{equation*} [f]_{L_w^p(I;X)}:={\sup_{r>0}\Big\{\big[r^p\lambda_f(r)\big]^{1/p}\Big\}}<\infty.\end{equation*}

\end{definition}

From the definition of $L^p_w(I;X)$ and Chebyshev's inequality (see \cite[Remark 4]{CarFe1}), we obtain the following classical result.

\begin{theorem}\label{inclusoeslp}
\begin{enumerate}
  \item[(i)] For $p\in[1,\infty)$, it holds that $L^p(I;X)$ is a proper vector subspace of $L^p_w(I;X)$ and 
  $$[f]_{L^p_w(I;X)}\leq\|f\|_{L^p(I;X)},$$
  for every $f\in L^p(I;X)$.\vspace*{0.2cm}
  \item[(ii)]  If $1\leq p< q<\infty$, it holds that $L_w^q(t_0,t_1;X)\subset L^p(t_0,t_1;X)$. Moreover, we have that
$$\|f\|_{L^p(t_0,t_1;X)}\leq\left(\dfrac{q}{q-p}\right)^{1/p}(t_1-t_0)^{(q-p)/(pq)}[f]_{L^q_w(t_0,t_1;X)},$$
for every $f\in L^q_w(t_0,t_1;X)$. 
\end{enumerate}
\end{theorem}
Finally, we present the improvement of \cite[Corollary 1]{CarFe1}.
\begin{theorem}\label{comp01} If $r\in \big(1/(1-\alpha),\infty\big]$ we have that $J_{t_0,t}^\alpha\big(L^1(t_0,t_1;X)\big)\not\subset L_w^r(t_0,t_1;X)$.
\end{theorem}

\begin{proof} Assume that there exists $r \in \big(1/(1-\alpha), \infty\big]$ such that $J_{t_0,t}^\alpha\big(L^1(t_0,t_1; X)\big) \subset L_w^r(t_0,t_1; X)$. Since for $m > l$, Theorem \ref{inclusoeslp} ensures that $L_w^m(t_0,t_1; X) \subset L^l(t_0,t_1; X)$, it follows that 
$$J_{t_0,t}^\alpha\big(L^1(t_0,t_1; X)\big) \subset L_w^r(t_0,t_1; X) \subset L^{1/(1-\alpha)}(t_0,t_1; X),$$
which contradicts the conclusions of \cite[Theorem 7]{CarFe1}.
\end{proof}

\subsection{Improving {\color{black} \cite[Theorem 10]{CarFe2}}}\label{sub02} In this previous work, we established that the RL fractional integral of order $\alpha \geq 1$ is continuous from $L^1(t_0,t_1; X)$ to $W_{RL}^{\alpha,1}(t_0,t_1;X)$. However, this result can be improved. To do so, let us first present the definition of $W_{RL}^{\alpha,1}(t_0,t_1;X)$, which was originally introduced by Carbotti and Comi in \cite{CaCo1} (see also \cite{BoId1,IdWa1}) and extended to vector-valued functions in \cite{CarFe2}.

\begin{definition}\label{sobolevriemann11} 
For $\alpha > 0$, we use the symbol $W_{RL}^{\alpha,1}(t_0,t_1; X)$ to denote the subspace of $W^{[\alpha]-1,1}(t_0,t_1; X)$ consisting of all functions $f:[t_0,t_1]\rightarrow X$ such that $D_{t_0,t}^\alpha f \in L^1(t_0,t_1; X)$. By considering the norm
\begin{equation*}\|f\|_{W_{RL}^{\alpha,1}(t_0,t_1;X)}:=\|f\|_{W^{[\alpha]-1,1}(t_0,t_1;X)}+\|D_{t_0,t}^\alpha f\|_{L^1(t_0,t_1;X)},\end{equation*}
the set $W_{RL}^{\alpha,1}(t_0,t_1;X)$ becomes a Banach space. We call $W_{RL}^{\alpha,1}(t_0,t_1;X)$ the RL fractional Bochner-Sobolev spaces.
\end{definition}

Now let us present the main theorem of this subsection.

\begin{theorem}
\label{consobfrac}
For $\alpha\in[1,\infty)$ and $\gamma\in(0,\alpha]$ we have that $J_{t_0,t}^\alpha:L^1(t_0,t_1;X)\rightarrow W_{RL}^{\gamma,1}(t_0,t_1;X)$ is a bounded linear operator. Moreover, it holds that
$$\|J_{t_0,s}^{\alpha}f\|_{W_{RL}^{\gamma,1}(t_0,t_1; X)} \leq \left[\sum_{j=0}^{[\gamma]-1} \dfrac{(t_1-t_0)^{\alpha-j}}{\Gamma(\alpha-j+1)}+\dfrac{(t_1-t_0)^{\alpha-\gamma}}{\Gamma(\alpha-\gamma+1)}\right]\|f\|_{L^1(t_0,t_1; X)},$$
for every $f\in L^1(t_0,t_1;X)$.
\end{theorem}

\begin{proof} Note that:

\begin{itemize}
    \item[(i)] For $\gamma \notin \mathbb{N}^\ast$, note that $D_{t_0,t}^\gamma \left[J_{t_0,t}^\alpha f(t)\right] = J_{t_0,t}^{\alpha - \gamma} f(t)$ for almost every $t \in [t_0, t_1]$. Together with \cite[Theorem 3.1]{CarFe0}, this implies that
    $$\big\|D_{t_0,t}^\gamma \left[J_{t_0,t}^\alpha f\right]\big\|_{L^1(t_0,t_1; X)} \leq \frac{(t_1 - t_0)^{\alpha - \gamma}}{\Gamma(\alpha - \gamma + 1)} \|f\|_{L^1(t_0,t_1; X)}.$$
    Then, through direct computation and by using \cite[Theorem 3.1]{CarFe0}, we obtain
    \begin{multline*}\qquad\big\|J_{t_0,s}^{\alpha}f\big\|_{W_{RL}^{\gamma,1}(t_0,t_1; X)}=\sum_{j=0}^{[\gamma]-1}\big\|J_{t_0,t}^{\alpha-j}f\big\|_{L^1(t_0,t_1; X)} + \big\|D_{t_0,t}^\gamma[J_{t_0,t}^\alpha f]\big\|_{L^{1}(t_0,t_1; X)}\\\leq \left[\sum_{j=0}^{[\gamma]-1} \dfrac{(t_1-t_0)^{\alpha-j}}{\Gamma(\alpha-j+1)}+\dfrac{(t_1-t_0)^{\alpha-\gamma}}{\Gamma(\alpha-\gamma+1)}\right]\|f\|_{L^{1}(t_0,t_1; X)}.\end{multline*}
    \item[(ii)] For $\gamma \in \mathbb{N}^\ast$, \cite[Theorem 3.1]{CarFe0} ensures
    $$\qquad\big\|J_{t_0,s}^{\alpha}f\big\|_{W_{RL}^{\gamma,1}(t_0,t_1; X)}=\sum_{j=0}^{\gamma}\big\|J_{t_0,t}^{\alpha-j}f\big\|_{L^1(t_0,t_1; X)}\\\leq \sum_{j=0}^{\gamma} \dfrac{(t_1-t_0)^{\alpha-j}}{\Gamma(\alpha-j+1)}\|f\|_{L^{1}(t_0,t_1; X)}.$$

\end{itemize}
\end{proof}

\subsection{RL fractional integral and essentially bounded functions} \label{sub03}

To conclude this section, we now address the case $p = \infty$ and examine the image of the RL fractional integral of order $\alpha>0$ when applied to the space $L^\infty(t_0,t_1;X)$.

\begin{theorem}\label{infcase1}
If $\alpha\in(0,1)$, then $J_{t_0,t}^\alpha:L^\infty(t_0,t_1;X)\rightarrow H^{0,\alpha}(t_0,t_1;X)$ is a bounded linear operator. Moreover, for any $r\in (\alpha,1)$ it holds that $J_{t_0,t}^\alpha\big(L^\infty(t_0,t_1;X)\big)\not\subset H^{0,r}(t_0,t_1;X)$.
  
\end{theorem}

\begin{proof} Due to Theorem \ref{continitial01}, we already know that $J_{t_0,t}^\alpha f \in C([t_0,t_1];X)$ and that
\begin{equation}\label{novaeq03112}
\sup_{s\in[t_0,t_1]}\|J^\alpha_{t_0,s}f(s)\|_X\leq K\|f\|_{L^\infty(t_0,t_1;X)}.
\end{equation}
for some constant $K>0$. If we consider $t \in (t_0,t_1]$, we have (recall inequality \eqref{new01})
\begin{multline}\label{novaeq0311222}
\big\|J^\alpha_{t_0,t}f(t)-J^\alpha_{t_0,w}f(w)\big|_{w=t_0}\big\|_X \leq \dfrac{1}{\Gamma(\alpha)} \int_{t_0}^t (t-s)^{\alpha-1} \|f(s)\|_X \, ds \\
\leq \left[\dfrac{(t-t_0)^\alpha}{\Gamma(\alpha+1)}\right] \|f\|_{L^\infty(t_0,t_1; X)}.
\end{multline}

Now, let $t,w \in (t_0,t_1]$ and, without loss of generality, assume that $t > w$. Then
\begin{multline*}
\big\|J^\alpha_{t_0,t}f(t)-J^\alpha_{t_0,w}f(w)\big\|_X \leq \dfrac{1}{\Gamma(\alpha)} \int_w^t (t-s)^{\alpha-1} \|f(s)\|_X \, ds \\
+ \dfrac{1}{\Gamma(\alpha)} \int_{t_0}^w \big|(w-s)^{\alpha-1} - (t-s)^{\alpha-1}\big| \|f(s)\|_X \, ds =: \mathcal{H} + \mathcal{I}.
\end{multline*}

At first, observe that H\"older's inequality ensures that
$$
\mathcal{H} \leq \dfrac{1}{\Gamma(\alpha)} \left[\int_w^t (t-s)^{\alpha-1} \, ds \right] \|f\|_{L^\infty(t_0,t_1; X)} = \left[\dfrac{(t-w)^\alpha}{\Gamma(\alpha+1)}\right] \|f\|_{L^\infty(t_0,t_1; X)}.
$$
while, if $\alpha<1$ (the case $\alpha\geq1$ is analogous), H\"older's inequality gives
\begin{multline*}
\mathcal{I} \leq \dfrac{1}{\Gamma(\alpha)} \left\{ \int_{t_0}^w \big[(w-s)^{\alpha-1} - (t-s)^{\alpha-1}\big] \, ds \right\} \|f\|_{L^\infty(t_0,t_1; X)} \\
= \left[\dfrac{(t-w)^\alpha + (w-t_0)^\alpha - (t-t_0)^\alpha}{\Gamma(\alpha+1)}\right] \|f\|_{L^\infty(t_0,t_1; X)}.
\end{multline*}

Therefore, we have that
\begin{equation}\label{novaeq0312121}
\big\|J^\alpha_{t_0,t}f(t)-J^\alpha_{t_0,w}f(w)\big\|_X \leq \left[\dfrac{2(t-w)^\alpha + (w-t_0)^\alpha - (t-t_0)^\alpha}{\Gamma(\alpha+1)}\right] \|f\|_{L^\infty(t_0,t_1; X)},
\end{equation}
for any $t,w \in (t_0,t_1]$. Since there exists $M = M(\alpha,t_0,t_1)>0$ such that
\begin{equation}\label{novaeq031}
\dfrac{2(t-w)^\alpha + (w-t_0)^\alpha - (t-t_0)^\alpha}{\Gamma(\alpha+1)} \leq M(t-w)^\alpha,
\end{equation}
for any $t,w \in (t_0,t_1]$, we have from \eqref{novaeq03112}, \eqref{novaeq0311222}, \eqref{novaeq0312121}, and \eqref{novaeq031} that there exists $N = N(\alpha,t_0,t_1)>0$ such that
\begin{equation*}
\|J^\alpha_{t_0,s}f\|_{H^{0,\alpha}([t_0,t_1]; X)} \leq N \|f\|_{L^\infty(t_0,t_1; X)}.
\end{equation*}

Finally, note that any nontrivial constant function serves as the desired counterexample for the non-inclusion stated at the end of the theorem.

\end{proof}

Now we present the general case.

\begin{theorem}\label{infcase2}
  Assume that $\alpha \geq 1$.
   \begin{itemize}
   \item[(i)] If $\alpha\not\in \mathbb{N}$ and $q \in \big(0, \alpha+1-[\alpha]\big]$, then $J_{t_0,t}^\alpha: L^\infty(t_0,t_1; X) \rightarrow H^{[\alpha]-1,q}([t_0,t_1]; X)$ is a bounded linear operator. Moreover, for any $r \in \big(\alpha+1-[\alpha], 1\big)$, we have that $J_{t_0,t}^\alpha \big(L^\infty(t_0,t_1; X)\big) \not\subset H^{[\alpha]-1,r}([t_0,t_1]; X)$.\vspace*{0.2cm}
   \item[(ii)] If $\alpha \in \mathbb{N}$, then $J_{t_0,t}^\alpha: L^\infty(t_0,t_1; X) \rightarrow W^{\alpha,\infty}(t_0,t_1; X)$ is a bounded linear operator.
   \end{itemize}
\end{theorem}

\begin{proof} The proof of this result follows the same steps as in Theorem \ref{lebsguecontinuity5}. However, instead of applying \cite[Theorem 3.1]{CarFe0} and {\color{black}\cite[Theorem 31]{CarFe2}}, we use Theorem \ref{continitial01} (adapted to $L^\infty(t_0,t_1; X)$ as done in estimate \eqref{novaeq03112}) and Theorem \ref{infcase1}. Therefore, we omit the proof.
\end{proof}

\section*{Acknowledgement}
 The second author was supported by CAPES/PRAPG grant
nº 88881.964878/2024-01 and Fundação de Amparo à Pesquisa e Inovação do Espírito
Santo (FAPES) under grant number T.O. 951/2023.

\end{document}